\documentclass{amsart}
\usepackage{amssymb,amsmath,amsthm,amsfonts,amsopn,url,bbm}
\usepackage{enumitem}
\usepackage{amscd, hyperref}
\theoremstyle{plain}
\newtheorem{thm}{Theorem}[section]

\newtheorem{prop}[thm]{Proposition}
\newtheorem{lemma}[thm]{Lemma}

\theoremstyle{definition}
\newtheorem{defn}[thm]{Definition}
\newtheorem*{defn*}{Definition}
\newtheorem{example}[thm]{Example}
\newtheorem{construction}[thm]{Construction}
\theoremstyle{remark}
\newtheorem{rmk}[thm]{Remark}
\newtheorem*{rmk*}{Remark}
\newcommand{\field}[1]{\mathbbm{#1}}
\newcommand{\N}{\field{N}}

\newcommand{\ideal}[1]{\mathfrak{#1}}
\newcommand{\m}{\ideal{m}}
\newcommand{\n}{\ideal{n}}
\newcommand{\p}{\ideal{p}}
\newcommand{\q}{\ideal{q}}
\newcommand{\ia}{\ideal{a}}
\newcommand{\ib}{\ideal{b}}
\newcommand{\ie}{\ideal{e}}
\newcommand{\ih}{\ideal{h}}
\newcommand{\ic}{\ideal{c}}

\newcommand{\func}[1]{\mathrm{#1} \,}

\newcommand{\Spec}{\func{Spec}}

\newcommand{\height}{\func{ht}}

\DeclareMathOperator{\Ass}{Ass}

\newcommand{\im}{\func{im}}

\newcommand{\ra}{\rightarrow}
\newcommand{\eHK}{e_{\rm HK}}

\newcommand{\qjj}{f}
\newcommand{\qjjm}{f^-}
\newcommand{\qjjp}{f^+}

\newcommand{\rjj}[3]{u_{#1}({#2}, {#3})}
\newcommand{\sjj}[3]{v_{#1}({#2}, {#3})}

\newcommand{\rjjm}[3]{u^-_{#1}({#2}, {#3})}
\newcommand{\sjjm}[3]{v^-_{#1}({#2}, {#3})}
\newcommand{\tjjm}[3]{w^-_{#1}({#2}, {#3})}
\newcommand{\ujjm}[3]{x^-_{#1}({#2}, {#3})}
\newcommand{\rjjp}[3]{u^+_{#1}({#2}, {#3})}
\newcommand{\sjjp}[3]{v^+_{#1}({#2}, {#3})}
\newcommand{\tjjp}[3]{w^+_{#1}({#2}, {#3})}
\newcommand{\vjjm}[3]{y^-_{#1}({#2}, {#3})}
\newcommand{\vjjp}[3]{y^+_{#1}({#2}, {#3})}
\DeclareMathOperator{\len}{\lambda}

\DeclareMathOperator{\ann}{ann}
\DeclareMathOperator{\Hom}{Hom}
\DeclareMathOperator{\Min}{Min}
\newcommand{\cM}{\mathcal{M}}
\newcommand{\cH}{\mathcal{H}}

\DeclareMathOperator{\lt}{\mathtt{lt}}
\DeclareMathOperator{\lm}{\mathtt{lm}}
\DeclareMathOperator{\lcm}{\mathtt{lcm}}

\author{Neil Epstein}
\address{Universit\"at Osnabr\"uck \\ 
Institut f\"ur Mathematik \\ 
49069 Osnabr\"uck \\ Germany}
 \email{nepstein@uni-osnabrueck.de}

\author{Yongwei Yao}
\address{
Department of Math and Statistics\\
Georgia State University\\
30 Pryor St., Atlanta, GA 30303}
\email{yyao@gsu.edu}
\thanks{The second author was partially supported by the National Science Foundation DMS-0700554}

\title{Some extensions of Hilbert-Kunz multiplicity}

\date{\today}
\begin{document}

\begin{abstract}
Let $R$ be an excellent Noetherian ring of prime characteristic.  Consider an arbitrary nested pair of ideals (or more generally, a nested pair of submodules of a fixed finite module).  We do \emph{not} assume that their quotient has finite length.  In this paper, we develop various sufficient numerical criteria for when the tight closures of these ideals (or submodules) match.  For some of the criteria we only prove sufficiency, while some are shown to be equivalent to the tight closures matching.  We compare the various numerical measures (in some cases demonstrating that the different measures give truly different numerical results) and explore special cases where equivalence with matching tight closure can be shown.  All of our measures derive ultimately from Hilbert-Kunz multiplicity.
\end{abstract}


\maketitle

\section{Introduction}
The classical notions of the Hilbert-Samuel function and multiplicity of a finite colength ideal have far-reaching implications in commutative algebra.  They arose in (and have many strong links to) intersection theory, and the multiplicity may be used to characterize when a pair $J \subseteq I$ of ideals have the same integral closure.  In characteristic $p$ algebra, these have natural analogues, namely the Hilbert-Kunz function and Hilbert-Kunz multiplicity of a finite colength ideal.  Both multiplicities may be extended somewhat to modules and to relative situations.  The two share many properties with each other.  It is notable that Hilbert-Kunz multiplicity characterizes when a pair $J \subseteq I$ of ideals have the same \emph{tight} closure, and that the Hilbert-Kunz function was recently used \cite{BM-unloc} to show that tight closure does not commute with localization, giving a negative answer to a very important question.  It should also be noted that the Hilbert-Kunz function is also linked to intersection theory (\emph{e.g.}, \cite{Kur-sRR}).

Achilles and Manaresi \cite{AchMan-jmult, AchMan-bimult} defined j-multiplicity of arbitrary ideals, extending the definition of Hilbert-Samuel multiplicity, in the context of the St\"uckrad-Vogel intersection algorithm in intersection theory (cf. the book \cite{FOV}).  Flenner and Manaresi \cite{FleMan-jmult} later showed that j-multiplicity may be used (through localizations) to characterize when an arbitrary pair $J \subseteq I$ have the same integral closure.

In this article, we explore a variety of techniques, all of which extend Hilbert-Kunz multiplicity and most of which involve 0th local cohomology, designed to provide criteria for when a nested pair of arbitrary ideals shares the same tight closure.

One possible approach would be: for each ideal, define a limit (or at least a finite limsup) based on the definition of j-multiplicity but using bracket powers in place of ordinary powers of ideals.  Such an approach would require that the numbers $\lambda_R(H^0_\m(R/I^{[q]})) / q^d$ (where $q$ varies over powers of $p$) be bounded above by a constant.  Such a result has proved elusive even when $R$ is essentially of finite type over a field and $R/I$ has small dimension (cf. \cite[Corollary 5.2]{Ab-QGor} for a solution to the already difficult dimension 1 case).  Thus, we limit ourselves here to \emph{relative} measures for a nested pair of ideals or submodules, where vanishing will be the benchmark for expecting tight closures to coincide.

For a Noetherian local ring $(R,\m)$, Monsky \cite{Mon-HK} defined the Hilbert-Kunz multiplicity of an $\m$-primary ideal $I$ via: \[
\eHK(I) := \lim_{q \ra \infty} \frac{\len(R/I^{[q]})}{q^d},
\]
where $d = \dim R$.  He showed that this is always well-defined, finite, and $\geq 1$ for any $\m$-primary ideal.

As noted above, Hilbert-Kunz multiplicity characterizes when a nested pair of $\m$-primary ideals has the same tight closure.  In fact, more is true (due to Hochster and Huneke): 
\begin{thm}\label{thm:HHlen}\cite[Theorem 8.17]{HHmain}
Let $(R,\m)$ be a Noetherian local ring of dimension $d$, and let $L \subseteq M \subseteq N$ be finitely generated $R$-modules such that $\len(M/L)<\infty$.  \begin{enumerate}[label=(\alph*)]
\item If $M \subseteq L^*_N$, then $\len(M^{[q]}_N / L^{[q]}_N) \leq C q^{d-1}$ for all $q$, for some constant $C$ independent of $q$.  Hence,
\[
\limsup_{q \ra \infty}  \frac{\len(M^{[q]}_N / L^{[q]}_N)}{q^d} = 0.
\]
\item Suppose that $R$ has a completely stable weak test element $c$, and that $\hat{R}$ is equidimensional and reduced\footnote{We remark here that by the methods used in proving our Theorem~\ref{thm:zerotc}, the assumption that $\hat{R}$ is reduced is unnecessary.}.  If \[
\liminf_{q \ra \infty}  \frac{\len(M^{[q]}_N / L^{[q]}_N)}{q^d} = 0,
\] then $M \subseteq L^*_N$.
\end{enumerate}
\end{thm}
If one sets $N := R$ and $J \subseteq I$ are $\m$-primary ideals, then this theorem gives the result on Hilbert-Kunz multiplicities as a special case.

In Section~\ref{sec:rel}, we introduce the limit $\rjj NLM$ for an arbitrary triple $L \subseteq M \subseteq N$ of finite $R$-modules where $(R,\m)$ is local (as well as the associated limsup $\rjjp NLM$ and liminf $\rjjm NLM$), based partly on the ideas of Theorem~\ref{thm:HHlen} and partly on j-multiplicity.  In Theorem~\ref{thm:zerotc}, we show that vanishing of the liminf version gives a one-way implication for $M \subseteq L^*_N$.

In Section~\ref{sec:examples}, we exhibit several situations where a strong converse holds.  Section~\ref{sec:gb} is devoted to one such example, where we resort to using Gr\"obner bases as an aid to computation.

In Section~\ref{sec:var}, we give several variants of the notion from Section~\ref{sec:rel}, and we show in Theorem~\ref{thm:var} that in cases where tight closure commutes with localization for $L \subseteq N$, a certain numerical vanishing condition is \emph{equivalent} to $M \subseteq L^*_N$.  We then show in Proposition~\ref{prop:reg} that two of these notions of relative Hilbert-Kunz multiplicity are, in general, quite distinct.

In Section~\ref{sec:Nak*}, we provide a numerical criterion, based on a tight closure variant of the Nakayama lemma previously proved by the first named author, which determines exactly when a pair of nested submodules have the same tight closure.  However, this criterion does not involve local cohomology, and so is further removed from j-multiplicity-type notions.

Finally in Section~\ref{sec:jHK}, we introduce a notion that looks more closely related to j-multiplicity than any of our other definitions.  In Theorem~\ref{thm:jHKcrit}, we show that it gives another criterion that determines exactly when two nested ideals $J \subseteq I$ with $\len(I/J)<\infty$ have the same tight closure.  Theorem~\ref{thm:jHKglobal} is a global version of this for rings which are F-regular on the punctured spectrum.

To conclude this introductory section, we recall some standard definitions and fix some notational conventions:

All rings are Noetherian and have prime characteristic $p>0$.  For a nonnegative integer $e$, denote $q :=p^e$.  All $R$-modules are considered as left modules unless noted otherwise.  For an $R$-module $M$, let ${}^eM$ be the ($R$-$R$)-bimodule which equals $M$ as an abelian group, and whose bimodule structure is given by $r \cdot z \cdot s := r^{p^e}sz$ for $z\in {}^eM$ and $r,s \in R$.  For an $R$-module $M$, $F^e(M)$ denotes the \emph{right} $R$-module structure on $M \otimes_R{}^eR$, but considered as a \emph{left} $R$-module.  Hence, ${}^e(F^e(M)) \cong M \otimes_R {}^eR$ as left $R$-modules.

Let $L \subseteq M$ be $R$-modules.  For $z\in M$, $z^q_M$ denotes the image $z \otimes{}^e1$ of $z$ under the map $M \ra F^e(M) = M \otimes_R{}^eR$.  Similarly, $L^{[q]}_M$ is the image of the map $F^e(L) \ra F^e(M)$ which is induced by the inclusion $L \hookrightarrow M$.  The \emph{tight closure} of $L$ in $M$, denoted $L^*_M$, is the submodule of $M$ consisting of all elements $z\in M$ such that there exists an element $c\in R^\circ$ (\emph{i.e.}, not in any minimal prime of $R$), possibly dependent on $z$, such that for all $q\gg 0$, we have $c z^q_M \in L^{[q]}_M$.  If $M=R$, we omit the subscript.  A ring $R$ is \emph{weakly F-regular} if all ideals are tightly closed (as $R$-submodules of $R$.)  $R$ is \emph{F-regular} if $R_\p$ is weakly F-regular for all $\p \in \Spec R$.  An element $c\in R^\circ$ is a ($q_0$-)\emph{weak test element} if there is some power $q_0$ of $p$ such that for all ideals $I$ and all $x\in R$, $x\in I^*$ if and only if $c x^q \in I^{[q]}$ for all $q\geq q_0$.  If we can take $q_0=1$, we say $c$ is a \emph{test element}.  For a more thorough review of these concepts, we recommend the seminal paper \cite{HHmain} and the monograph \cite{HuTC}.

\begin{rmk}
In many places in this manuscript, the reader will encounter the phrase: ``let $L \subseteq M \subseteq N$ be finite(ly generated) $R$-modules.''  We state things this way for convenience, and because in many situations what one is interested in is the case where $N$ is a ring and $L$, $M$ are ideals.  However, \emph{wherever this phrase occurs}, note that it is not necessary to assume that \emph{any} of the modules in question are finite, but only that the quotient module $N/L$ is finite.  A reader who is interested in possibly infinite modules should keep this in mind.
\end{rmk}

\section{Relative multiplicity}\label{sec:rel}
Throughout, let $R$ be a Noetherian ring of prime characteristic $p>0$.

When $(R,\m)$ is local and $M$ an $R$-module, we let \[
\Gamma_\m(M) := \{ z \in M \mid \m^n z = 0 \text{ for some } n \in \N \}.
\]

Recall that $\Gamma_\m$ is a left-exact functor, and that a finitely generated $R$-module $M$ has finite length if and only if $M=\Gamma_\m(M)$.  We start with the following definition:

\begin{defn}\label{def:rjj}
Let $L \subseteq M \subseteq N$ be $R$-modules, where $(R,\m)$ is local of dimension $d$.  Then the \emph{relative multiplicity of $L$ against $M$ (in $N$)} is \[
\rjjp N L M := \limsup_{q \ra \infty} \frac{\len(\Gamma_\m(M^{[q]}_N / L^{[q]}_N))}{q^d}.
\]
(resp. \[
\rjjm N L M := \liminf_{q \ra \infty} \frac{\len(\Gamma_\m(M^{[q]}_N / L^{[q]}_N))}{q^d}.)
\]
If these are equal (\emph{i.e.}, the limit is well-defined), then the common number is written $\rjj NLM$.
\end{defn}
When $N$ is understood (\emph{e.g.}, when $N=R$), we omit it from the notation.

Hence, if $\len(M/L)<\infty$ and $N/L$ is finite, then Theorem~\ref{thm:HHlen} shows that under the conditions of that theorem, $M \subseteq L^*_N$ if and only if $\rjjp NLM = 0$ if and only if $\rjjm NLM=0$. (The assumption in Theorem~\ref{thm:HHlen} that $L \subseteq M \subseteq N$ are finitely generated is unnecessary.)

At this point, the reader may wonder the following: if $J \subseteq I$ are ideals with the same tight closure, and the ring is reasonable enough, can it happen that their quotient has infinite length?  In fact it can, as the following example shows, in which we ``add a variable'':

\begin{example}\label{ex:addvar}
Let $(A,\n)$ be \emph{any} Noetherian local ring that is not weakly F-regular, and let $\ib \subsetneq \ia$ be a pair of nested distinct ideals that have the same tight closure.  Let $z \in \ia \setminus \ib$.  Let $X$ be an indeterminate over $A$, let $Q := A[X]$, and let $R := Q_{\n Q + XQ}$.  Let $\m := \n R + X R$ be the unique maximal ideal of $R$.  Let $J := \ib R$ and $I := \ia R$.  Then $J^* = I^*$.  We have:
\[
\frac{I}{J} \supseteq \frac{J + Rz}{J} \cong \frac{Rz}{J \cap Rz} \cong \frac{R}{J :_R z}.
\]
Note that for all integers $n\geq 0$, we have $X^n z \notin J$.  Hence $X \notin \sqrt{(J :_R z)}$, which shows that $(J :_R z)$ is a non-$\m$-primary ideal.  By the containments displayed above, then, we have
\[
\len_R(I/J) \geq \len_R(R/(J :_R z)) = \infty.
\]
\end{example}

In the following theorem, we extend one of the implications of Theorem~\ref{thm:HHlen} to the infinite length case.  First, recall the following definition from \cite{HHexponent}, here generalized to the module case (see \cite[p. 4850]{nmepdep} for further explanation):

\begin{defn}
Let $L$ be a submodule of $N$ and $z\in N$.  Then $\q \in \Spec R$ is a \emph{stable prime associated to $L \subseteq N$ and $z$} if $z \notin (L_\q)^*_{N_\q}$, but for all primes $\p \subsetneq \q$, $z \in (L_\p)^*_{N_\p}$.  The set of all such primes is denoted $T^N_L(z)$, and we set \[
T^N_L := \bigcup_{z \in N} T^N_L(z).
\]
\end{defn}

\begin{thm}\label{thm:zerotc}
Let $R$ be a Noetherian ring, and let $L \subseteq M \subseteq N$ be $R$-modules such that $N/L$ is finitely generated.  Suppose that $R$ contains a completely stable weak test element $c$, and that $\widehat{R_\p}$ is equidimensional for all $\p \in T^N_L$.  If $M \nsubseteq L^*_N$, then for any $x \in M \setminus L^*_N$, we have $\rjjm{N_\p}{L_\p}{M_\p}>0$ for all $\p \in T^N_L(x)$.  (Hence if $\rjjm{N_\p}{L_\p}{M_\p}=0$ for all $\p \in T^N_L$, then $M \subseteq L^*_N$.)
\end{thm}

\begin{proof}
Let $x$ be as in the theorem.  We may assume that all modules are finite and that $M=L+Rx$.  By \cite[Prop 3.3(g)]{HHexponent}, we have $T^N_L(x) \neq \emptyset$.  Take any $\p \in T^N_L(x)$.  Let $Q := \{q=p^e \mid \p \text{ is minimal over } (L^{[q]}_N :_R c x^q_N)\}$.  By \cite[Prop 3.1]{HHexponent}, $\N \setminus \{e \mid p^e \in Q\}$ is finite.  Moreover, $R_\p / ((L^{[q]}_N)_\p :_{R_\p} c x^q_{N_\p})$ has finite length over $R_\p$ for all $q \in Q$.

Now we can localize at $\p$ and complete the ring.  Replacing $R$ by $\widehat{R_\p}$, the new ring $R$ is complete and equidimensional.

We have an exact sequence \[
0 \ra \frac{R}{L^{[q]}_N :_R c x^q_N} \stackrel{\cdot c}{\longrightarrow} \frac{R}{L^{[q]}_N :_R x^q_N}
\]
because $L^{[q]}_N :_R c x^q_N = (L^{[q]}_N :_R x^q_N) :_R c$ for all $q\in Q$ (hence for all $q\gg 0$).  Applying the left-exact functor $\Gamma_\m(-)$ to it, and using the fact that $\len(R / (L^{[q]}_N :_R c x^q_N))< \infty$, we get another exact sequence \[
0 \ra \frac{R}{L^{[q]}_N :_R c x^q_N} \stackrel{\cdot c}{\longrightarrow} \Gamma_\m\left(\frac{R}{L^{[q]}_N :_R x^q_N}\right).
\]
Hence, $\len(R/(L^{[q]}_N :_R c x^q_N)) \leq \len(\Gamma_\m(R/(L^{[q]}_N :_R x^q_N)))$ for all $q\gg 0$.  Therefore, to show the claim of the theorem, we need only show that $\len_R(R/(L^{[q]}_N :_R c x^q_N))$ is bounded below by a constant multiple of $q^d$.

Since tight closure can be checked modulo minimal primes, there is some minimal prime $\q$ of $R$ such that the image of $x$ is not in $(L + \q N / \q N)^*_{N/\q N}$ as $(R/\q)$-modules.  Since $ \dim (R/\q)=\dim R = d$ (by the assumption of equidimensionality), and since the length of the desired quotient can only decrease when the colon is computed modulo $\q$, we may replace $R$ by $R/\q$ and assume that $R$ is a complete local \emph{domain}.  After doing this, we note that by the Cohen structure theorem, $R$ is module-finite and torsion-free over a complete regular local ring $(A,\n)$, and by replacing $c$ by a multiple we may assume that $c \in A^\circ$.

Next, we see that \[
\frac{R}{L^{[q]}_N :_R c x^q_N} \cong \frac{L^{[q]}_N + R c x^q_N}{L^{[q]}_N}
\supseteq \frac{L^{[q]}_N + A c x^q_N}{L^{[q]}_N} \cong \frac{A}{L^{[q]}_N :_A c x^q_N}.
\]
In the above, the Frobenius computations are being done over $R$.  The first isomorphism is as $R$-modules, hence also as $A$-modules, and the second isomorphism is as $A$-modules.  However, the $A$-module length of an $R$-module is the same as its $R$-module length, so it makes sense (and is true) to say that $\len(A/(L^{[q]}_N :_A c x^q_N)) \leq \len(R/(L^{[q]}_N :_R c x^q_N))$ for all $q\gg 0$.

Since $x \notin L^*_N$, then by the last paragraph of the proof of \cite[Theorem 8.17]{HHmain}, there is some power $q_1$ of $p$ such that \[
L^{[q]}_N :_A x^q_N \subseteq \n^{[q/q_1]}
\]
for all $q \gg 0$.  Thus, \[
L^{[q]}_N :_A c x^q_N = (L^{[q]}_N :_A x^q_N) :_A c \subseteq \n^{[q/q_1]} :_A c,
\]
which implies that $\len(A/(\n^{[q/q_1]} :_A c)) \leq \len(A/(L^{[q]}_N :_A c x^q_N))$.

Next, we have the following short exact sequence of $A$-modules: \[
0 \ra \frac{A}{\n^{[q/q_1]} :_A c} \stackrel{\cdot c}{\longrightarrow} \frac{A}{\n^{[q/q_1]}}
\ra \frac{\bar{A}}{\n^{[q/q_1]} \bar{A}} \ra 0,
\]
where $\bar{A} := A/cA$.  Combining the length equality we get from this sequence with the inequalities we have thus far, we have for all $q \in Q$:  \begin{align*}
 \len \left( \frac{R}{L^{[q]}_N :_R cx^q_N} \right) &\geq \len \left( \frac{A}{L^{[q]}_N :_A c x^q_N} \right) 
\geq \len \left( \frac{A}{\n^{[q/q_1]} :_A c}\right)\\
&= \len \left(\frac{A}{\n^{[q/q_1]}}\right) - \len \left( \frac{\bar{A}}{\n^{[q/q_1]}\bar{A}} \right).
\end{align*}

Dividing the difference in the last line by $q^d$ and taking the limit as $q$ approaches infinity, we get $1/q_1^d$ which, as required, is positive.
\end{proof}

\begin{rmk*}
By \cite[Proposition 3.3(a)]{HHexponent} (in light of \cite[footnote 6]{nmepdep}), if tight closure commutes with localization for the inclusion $L \subseteq N$, then $T^N_L = \Ass_R(N/L^*_N)$, which is a finite set.  So in this situation, only finitely many primes need to be checked in order to use the test in Theorem~\ref{thm:zerotc}.
\end{rmk*}

\section{Relative multiplicity: special cases}\label{sec:examples}
In this section, we give conditions under which a converse to Theorem~\ref{thm:zerotc} holds, so that we have a criterion that determines exactly when two submodules share the same tight closure.

For a first example, we note that if $R$ is weakly F-regular at all non-maximal primes, then $L \subseteq M \subseteq L^*_N$ implies that $M_\m / L_\m$ has finite length for all maximal ideals $\m$, so that Theorem~\ref{thm:HHlen} then yields the converse to Theorem~\ref{thm:zerotc} in this case.

\subsection*{Relative multiplicity and finite projective dimension}
Let $J$ be an ideal of finite projective dimension with no embedded primes (\emph{e.g.}, a parameter ideal in a Cohen-Macaulay local ring),  and $I$ an ideal such that $J \subseteq I \subseteq J^*$.  Then $\rjj{R_\p}{J_\p}{I_\p}= 0$ for all prime ideals $\p$.

To see this, let $G.$ be a finite projective resolution of $R/J$.  Then for any $q=p^e$, we have that $F^e(G.)$ is a projective resolution of $R/J^{[q]}$, by \cite{PS-thesis}.  Then since projective resolutions commute with localization and by the conditions for exactness of a complex \cite{BE-exact}, $J$ and $J^{[q]}$ have the same associated primes (namely, the minimal primes of $J$).  Thus we have \[
\Ass(I^{[q]} / J^{[q]}) \subseteq \Ass(R/J^{[q]}) = \Ass(R/J) = \Min(R/J).
\]

Hence for any prime $\p \notin \Min(R/J)$, we have $H^0_\p(I^{[q]} / J^{[q]}) = 0$ for all $q$, so that $\rjj{R_\p}{J_\p}{I_\p} = 0$.  On the other hand, for any $\p \in \Min(R/J)$, we have that $J_\p$ and $I_\p$ are primary to the maximal ideal of $R_\p$, so that for these primes, \[
\rjj{R_\p}{J_\p}{I_\p} = \eHK(J_\p) - \eHK(I_\p) = 0,
\]
with the last equality holding because $J_\p \subseteq I_\p \subseteq (J^*)_\p \subseteq (J_\p)^*$.

The same argument shows something slightly more general.  Namely,

\begin{prop}\label{pr:tczero-pdim}
Let $(R,\m)$ be a Noetherian local ring, and let $L \subseteq M \subseteq N$ be $R$-modules such that $N/L$ is a finite module of finite projective dimension with no embedded primes.  Then $\rjj {N_\p}{L_\p}{M_\p}$ is well-defined and finite for all $\p \in \Spec R$.  Moreover, $M \subseteq L^*_N$ if and only if $\rjj {N_\p}{L_\p}{M_\p}=0$ for all $\p \in \Spec R$.
\end{prop}

\begin{proof}
As usual, we may assume all the modules are finitely generated.  Then:
\[
0 \leq \rjj {N_\p}{L_\p}{M_\p} = \begin{cases}
0, & \p \notin \Min(N/L)\\
\displaystyle \lim_{q \ra \infty} \frac{\len_{R_\p}(M_\p^{[q]} / L_\p^{[q]})}{q^{\height \p}}, & \p \in \Min(N/L),\\
\end{cases}
\]
and hence if $M \subseteq L^*_N$, Theorem~\ref{thm:HHlen} applied to $R_\p$ shows that $\rjj {N_\p}{L_\p}{M_\p} =0$ for all $\p$.
\end{proof}

\subsection*{Relative multiplicity and finite F-representation type}
The concept of finite F-representation type is due to Smith and van der Bergh \cite{SmV-sim}.
\begin{defn}
Let $R$ be an F-finite Noetherian local ring.  It has \emph{finite F-representation type} (abbreviated \emph{FFRT}) if there is a finite set of finitely generated $R$-modules $M_1, \dotsc, M_s$ and integers $c_{i,e}$ for all $1 \leq i \leq s$ and all positive integers $e$ such that $\displaystyle {}^eR \cong \bigoplus_{i=1}^s M_i^{\oplus c_{i,e}}$ as $R$-modules, for all $e$.
\end{defn}

Smith and van der Bergh observed that the following classes of local F-finite rings have finite F-representation type: \begin{itemize}
\item regular rings
\item rings of finite Cohen-Macaulay type
\item any direct summand of a ring of finite F-representation type
\end{itemize}
On the other hand, they showed that the cubical cone $k[\![X,Y,Z]\!] / (X^3+Y^3+Z^3)$ does \emph{not} have FFRT.\footnote{Note that rings with FFRT need not be F-regular (or even F-rational).  Shibuta \cite{Sh-1dFFRT} proved that if $R$ is any 1-dimensional complete local domain of prime characteristic whose residue field is either finite or algebraically closed, then $R$ has finite F-representation type.}

The second named author took the study of such rings further.  From this point on, we fix the modules $M_1, \dotsc, M_s$ in the definition, and we assume that they are indecomposable, nonzero, and of distinct isomorphism classes.  For simplicity, we assume in the following that $R$ is complete, so that it satisfies the Krull-Schmidt condition and the numbers $c_{i,e}$ are uniquely determined. Let $a := [k : k^p]$.  Then
\begin{thm}\label{thm:Yao}\cite{Yao-FFRT}
Under the above circumstances, \begin{enumerate}
\item The limit \[
\ell_i := \lim_{e \ra \infty} \frac{c_{i,e}}{(ap^d)^e}
\]
is well-defined and finite for $1\leq i \leq s$.

\noindent{\rm (Without loss of generality, we assume from this point on that $\ell_i >0$ for $1\leq i \leq r$ and $\ell_i=0$ for $r<i\leq s$.  The modules $M_1, \dotsc, M_r$ are called the \emph{F-contributors}.)}

\item $r \geq 1$. (That is, there is at least one F-contributor.)

\item Let $U := \displaystyle \bigoplus_{i=1}^r M_i$.  For any finitely generated modules $L \subseteq N$, we have \[
L^*_N \subseteq \ker (N \ra \Hom_R(U, (N/L) \otimes_R U)),
\]
where the map is defined by $n \mapsto (u \mapsto \bar{n} \otimes u)$.  If $N=R$, this just means that $L^* \subseteq \ann_R(U/LU)$.

\item If $R$ has a completely stable test element and $\hat{R}$ is reduced\footnote{Here too, the reducedness assumption appears to be unnecessary, in light of methods used in the proof of Theorem~\ref{thm:zerotc}.} and equidimensional, then the displayed containment above becomes an equality (so that when $N=R$, $L^* = \ann_R(U/LU)$).
\end{enumerate}
\end{thm}

Here we compute $\rjj NLM$ when $R$ is complete and has FFRT:  

For any finitely generated $R$-module $Z$, and any $e$, we have ${}^e(F^e(Z)) \cong Z \otimes_R {}^eR$.  Using this and the fact that ${}^e(-)$ is an exact functor, if $R$ is complete we have \begin{align*}
{}^e (M^{[q]}_N / L^{[q]}_N) &= {}^e \ker(F^e(N/L) \ra F^e(N/M)) \\
&\cong \ker({}^e(F^e(N/L)) \ra {}^e(F^e(N/M))) \\
&\cong \ker ( (N/L) \otimes {}^eR \ra (N/M) \otimes {}^eR) \\
&= \bigoplus_{i=1}^s \ker( (N/L) \otimes_R M_i \ra (N/M) \otimes_R M_i)^{\oplus c_{i,e}}.
\end{align*}
Since $\Gamma_\m$ is left-exact and commutes with ${}^e(-)$, we have \[
{}^e \Gamma_\m(M^{[q]}_N / L^{[q]}_N) \cong  \bigoplus_{i=1}^s \Gamma_\m(\ker( (N/L) \otimes_R M_i \ra (N/M) \otimes_R M_i))^{\oplus c_{i,e}}.
\]
Set $K_i := \Gamma_\m(\ker( (N/L) \otimes_R M_i \ra (N/M) \otimes_R M_i))$ for $1 \leq i \leq s$.
Since $\len({}^e Z) = a^e \len(Z)$ for any finite length $R$-module $Z$, we have \begin{align*}
\rjj NLM &= \lim_{q \ra \infty} \frac{\len(\Gamma_\m(M^{[q]}_N / L^{[q]}_N))}{p^{de}}
=  \lim_{q \ra \infty} \frac{\len({}^e \Gamma_\m(M^{[q]}_N / L^{[q]}_N)}{(a p^d)^e} \\
&= \sum_{i=1}^s  \len(K_i) \lim_{e \ra \infty} \frac{c_{i,e}}{(a p^d)^e} = \sum_{i=1}^r \len(K_i) \ell_i.
\end{align*}
(The sum only goes to $r$, since the limits equal 0 for $r<i\leq s$.)

Now we can state a converse to Theorem~\ref{thm:zerotc} for rings with FFRT.
\begin{thm}\label{thm:tczero-FFRT}
Let $R$ be an F-finite ring with FFRT, and let $L \subseteq M \subseteq N$ be finitely generated $R$-modules.  If $M \subseteq L^*_N$, then $\rjj {N_\p}{L_\p}{M_\p} = 0$ for all $\p\in \Spec R$.
\end{thm}

\begin{proof}
Since tight closure persists in localizations \cite[Theorem 6.24]{HHbase}, and since the FFRT property localizes, we may assume that $(R,\m)$ is local and just show that $\rjj NLM=0$.

Let $M_i$, $s$, $r$, $\ell_i$, and $U$ be as in Theorem~\ref{thm:Yao}, and let $K_i$ be as in the discussion above.  Since $M \subseteq L^*_N$, part (3) of that theorem implies that $M \subseteq \ker (N \ra N/L \ra \Hom_R(U, (N/L) \otimes_R U))$.  Translating, this means that the natural map \[
(N/L) \otimes_R U \ra (N/M) \otimes_R U
\]
is an isomorphism, whence the natural maps $(N/L) \otimes_R M_i \ra (N/M) \otimes_R M_i$ are isomorphisms, and hence $K_i=\Gamma_\m(\ker ((N/L) \otimes_R M_i \ra (N/M) \otimes_R M_i)) =\Gamma_\m(0)=0$, for $1\leq i \leq r$.  Thus we have \[
\rjj NLM = \sum_{i=1}^r \len(K_i) \ell_i = 0. \qedhere
\]
\end{proof}

\subsection*{Relative multiplicity and finite generation of $R[x;f]$-modules}
One of the standard constructions in noncommutative ring theory is the \emph{skew polynomial ring} (cf. \cite[Example 1.7]{Lam-1}, or almost any other introductory text on noncommutative rings).  Given a ring $R$, an indeterminate $x$, and a ring endomorphism $f: R \ra R$, the skew polynomial ring $S := R[x; f]$ is an $\N$-graded $R$-algebra, which looks like $\bigoplus_{n\geq 0} R x^n$ as a graded $R$-module, with multiplication given by $(r x^n) (s x^m) := r f^n(s) x^{m+n}$ for $r, s \in R$.  When $R$ is a commutative Noetherian ring of prime characteristic $p$, this ring and its modules were first studied by Yuji Yoshino \cite{Yo-Frob}, and were studied much further, to great effect, by Lyubeznik \cite{Lyu-Fmod} (who called some of the $S$-modules `F-modules'), and by Rodney Sharp and his collaborators (\emph{e.g.}, \cite{NoSh, KaSh-Frob, Sha-HSL, Sha-grann}).  In particular, Sharp studied various right and left $R[x;f]$-module structures on top local cohomology modules of finite $R$-modules, obtaining striking results on parameter test elements.

Let $\cM :=\displaystyle \bigoplus_{e \geq 0} (M^{[p^e]}_N / L^{[p^e]}_N) X^e$, and $\cH := \Gamma_\m(\cM) = \displaystyle \bigoplus_{e \geq 0} \Gamma_\m(M^{[p^e]}_N / L^{[p^e]}_N) X^e$, where $L \subseteq M \subseteq N$ are fixed $R$-modules such that $M/L$ is finite.

Then $\cM$ is a graded left $R[x;f]$-module, with $R[x;f]$-action given by $x \cdot \overline{m} X^e := \overline{m^p} X^{e+1}$.  Moreover, it is finitely generated in degree $0$ by the $R$-generators of $M/L$, and $\cH$ is a graded left $R[x;f]$-submodule of $\cM$.  Since $R[x;f]$ is almost never left- (or right-) Noetherian \cite{Yo-Frob}, we cannot assume that an arbitrary left submodule of a finite left $R[x;f]$-module is finitely generated. However:

\begin{thm}\label{thm:tczero-fg}
Suppose $(R,\m)$ is local and $\cH$ is finitely generated as a left $R[x;f]$-module.  If $M \subseteq L^*_N$, then $\rjj NLM = 0$.
\end{thm}

\begin{proof}
We may immediately assume that $L=0$.

The hypotheses imply that there is some power $q_0$ of $p$ such that for all $q \geq q_0$, we have $\Gamma_\m(M^{[q]}_N) = H^{[q/q_0]}_{F^{e_0}(N)}$, where $H:=\Gamma_\m(M^{[q_0]}_N)$.   Then we have \[
\len(\Gamma_\m(M^{[q]}_N)) = \len(H^{[q/q_0]}_{F^{e_0}(N)}) \leq C (q/q_0)^{d-1} = (C/q_0^{d-1}) q^{d-1},
\]
where $C$ is a constant which is independent of $q$.  The last inequality is by Theorem~\ref{thm:HHlen}, since $H \subseteq M^{[q_0]}_N \subseteq 0^*_{F^{e_0}(N)}$. So $\rjj N0M =0$.
\end{proof}

\section{A specific example}\label{sec:gb}
Let $J \subseteq I$ be ideals with the same tight closure.  An analysis of the ideas surrounding Proposition~\ref{pr:tczero-pdim} yields the following observation: \emph{The critical situation occurs when there exist prime ideals $\p \subsetneq \m$ such that $\p, \m \in \Ass_R(I^{[q]}/J^{[q]})$ for infinitely many values of $q$.}  One may ask whether this can happen.  For instance, in Example~\ref{ex:addvar}, the critical situation does not occur for the ideals $J \subseteq I$ in $R$ unless it already was an issue for the ideals $\ib \subseteq \ia$ in $A$.  Indeed, for each $q$, there is a bijective correspondence between the sets $\Ass_A(\ia^{[q]}/\ib^{[q]})$ and $\Ass_R(I^{[q]}/J^{[q]})$, given by $\p \mapsto \p R$.

However, the situation outlined above can happen, as shown below.  Moreover, the expected converse to Theorem~\ref{thm:zerotc} holds, at least in the given example.  Note that we cannot get this result from Proposition~\ref{pr:tczero-pdim}, nor do we know how to obtain it from either Theorem~\ref{thm:tczero-FFRT} or Theorem~\ref{thm:tczero-fg}.  Therefore, we had to use computational methods.

Before we get to the specific characteristic $p$ situation, we give a somewhat more general construction, which works over any field, and may be of independent interest.  As we will be using Gr\"obner basis techniques, we set some notations and recall some facts:

\begin{defn}
Let $A$ be a polynomial ring, $>$ a monomial order, and $f\in A \setminus \{0\}$.  The expressions $\lt(f)$ and $\lm(f)$ denote, respectively, the leading term and the leading monomial of $f$ with respect to the given order.

Given two elements $f,g \in A \setminus \{0\}$, the \emph{S-polynomial} of $f$ and $g$ is given by \[
S(f,g) := \frac{\lcm(\lt(f), \lt(g))}{\lt(f)} \cdot f - \frac{\lcm(\lt(f), \lt(g))}{\lt(g)} \cdot g,
\]
where $\lcm$ means the least common multiple.
\end{defn}

The following theorem is a slightly nonstandard (albeit well-established) form of the \emph{Buchberger criterion}:

\begin{thm}\label{thm:Bucrit}\cite[Theorem 2.9.3]{CLO-book1e3}
Let $A$ be a polynomial ring over a field, let $>$ be a monomial order, and let $G = \{g_1, \dotsc, g_n\}$ be a finite subset of $A$.  Then $G$ is a Gr\"obner basis if and only if there exist elements $a_{ijk} \in A$ such that for each pair $(j,k)$ with $1\leq j <k \leq n$, we have \[
S(g_j,g_k) = \sum_{i=1}^n a_{ijk} g_i,
\]
in such a way that for each nonzero $a_{ijk}$, we have $\lm(S(g_j, g_k)) \geq \lm (a_{ijk} g_i)$ with respect to the given monomial order.
\end{thm}

\begin{thm}\label{thm:elim} \cite[Theorem 4.3.11 and the discussion which follows]{CLO-book1e3}
Let $A$ be a polynomial ring over a field $k$, let $I$ be an ideal of $A$ and $0 \neq u\in A$.  Let $r$ be an indeterminate over $A$, and let $B = A[r]$ be a polynomial ring, ordered with lexicographic order in such a way that $r>x$ for all variables $x$ of $A$.  Let $\ia := rIB + (1-r)uB \subseteq B$.  Then $\ia \cap A = I \cap (u)$, and if $F$ is a Gr\"obner basis of $\ia$ in $B$, then $\ia \cap A$ is generated by the set of elements of $F$ whose leading terms are not multiples of $r$.
\end{thm}

\begin{construction}\label{cons}
Let $k$ be an arbitrary field, let $m \in \N$ such that $m\geq 4$, and let $n=2m+1$. We also impose the condition that if $p$ is the characteristic of $k$, then $p \nmid m$, which is automatically satisfied if $p=0$. Let $A := k[s,x,y]$, $\m := (s,x,y) \subseteq A$, $g=xy(x-y)(x+y-sy)$, and $\ie := (x^n, y^n, g) \subseteq A$.  Let $f := \sum_{j=2}^{n-1} (-1)^j x^{n+1-j} y^j$.  Let $\ih := \ie + (f)$.  Let $\ib := (x,y)^{n+2}$.  Then we will show the following: \begin{enumerate}
\item $\ib \subseteq \ie$,
\item $sf \in \ie$ (hence, $\ih \subseteq (\ie :s)$),
\item $xf, yf \in \ie$ (hence, $\m \subseteq (\ie :f)$),
\item $f \notin \ie$ (hence, $(\ie:f) \neq A$, so that $\m = (\ie :f)$),
\item $\ih$ is $s$-saturated  (that is, $(\ih : s) = \ih$),
\item $\ie : \m^\infty = \ie : s^\infty = \ih$, and
\item $H^0_\m(A/\ie) \cong A/\m$.
\end{enumerate}

To see (1), take a typical monomial generator $x^i y^j$ of $\ib$.  That is, $i+j = n+2$.  Since $x^n, y^n \in \ie$, we may assume that $1\leq j \leq n-1$, so that $i\geq 3$.  Note that modulo $g$, we have \[
x^3y \equiv s x^2 y^2 - (s-1) x y^3.
\]
Multiplying this by $x^{i-3} y^{j-1}$, we have $x^i y^j \in (x^{i-1} y^{j+1}, x^{i-2} y^{j+2}, g)$.  Then apply induction to obtain $x^i y^j \in (x^2 y^n, x y^{n+1}, g) \subseteq \ie$.

To see (2), note that modulo $g$, we have \[
s x y^2 (x-y) \equiv xy (x^2 - y^2).
\]
Using this congruence, we have:
\begin{align*}
s(f - xy^n) &= s x y^2(x-y)\left(\sum_{j=0}^{m-1} x^{n-3-2j} y^{2j}\right) \\
&\equiv xy(x^2 -y^2)\left(\sum_{j=0}^{m-1} x^{n-3-2j} y^{2j}\right) \\
&= xy(x^{2m} - y^{2m}) = x^{n} y - x y^{n}.
\end{align*}
Thus, $sf \in (x^n, y^n, g) = \ie$, as required.

To see (3), let $t=s-1$.  Modulo $g$, we have the equivalence \[
tx y^2 (x-y) \equiv x^2 y (x-y).
\]
It follows by induction (on $i$) that for all integers $i\geq 1$, $a\geq 1$, and $b\geq i+1$, we have $t^{i} x^a y^b (x-y) \equiv x^{a+i} y^{b-i} (x-y)$ (modulo $g$).  In particular (letting $a=1$ and $b=n$), for all $1\leq i\leq n-1$, we have \[
t^i x y^n(x - y)\equiv x^{i+1}y^{n-i} (x-y).
\]
modulo $g$. Note also that $-x f+ x^{n} y^2 =yf - x^2 y^n = \sum_{j=1}^{m-1} x^{2j+1} y^{n-2j}(x-y)$.  But by the above (since $2(m-1) = n-1$), this latter sum is congruent (modulo $g$) to $y^n \cdot \left(\sum_{j=1}^{m-1} t^{2j} x (x-y) \right)$.  Thus, $-xf, yf \in (x^n, y^n, g) = \ie$, as required.

In order to demonstrate (4), we require the introduction of Gr\"obner bases into the discussion.  From now on, we will use \emph{lexicographic}\footnote{We emphasize here that we are \emph{not} using degree-lexicographic order.  So for instance, in this ordering, we have $s>x^2$.  Indeed, $s>x^{200}$.} order, with $s>x>y$.  We claim that \[
G := \{g, x^n, x^{n-1} y^3, x^{n-2} y^4, \cdots, x^3 y^{n-1}, y^{n} \}
\] is a Gr\"obner basis of $\ie$ with respect to lex order.  First, since the elements of $G$ consists of the generating set $\{g, x^n, y^n\}$ of $\ie$ along with some elements of $\ib$ (an ideal which by (1) is contained in $\ie$), it follows that $G$ is indeed a generating set for $\ie$.  To show that it is a Gr\"obner basis, we shall find $a_{ijk}$ as in Theorem~\ref{thm:Bucrit}.  But since the S-polynomial of a pair of monomials is always 0, we only need to look at the S-polynomials $S(m,g)$ for monomials $m$ of $G$.  In the following list, we represent each S-polynomial in two ways.  First, we write it in lexicographic order, and then we write it in the form given by Theorem~\ref{thm:Bucrit}:  \begin{itemize}
\item $S(y^{n}, g) = -sxy^{n+1} - x^3y^{n-1} + xy^{n+1} = (-sxy)y^{n} - 1(x^3 y^{n-1}) + (xy)y^{n}$.
\item $S(x^3 y^{n-1}, g) = -s x^2 y^{n} - x^4 y^{n-2} + x^2 y^{n} = (-s x^2) y^{n} - 1(x^4 y^{n-2}) + (x^2) y^{n}$.
\item For any $i$ with $4 \leq i \leq n-2$, we have $x^{i-1} y^{n+3-i}, x^{i+1} y^{n+1-i} \in G$.  And \begin{align*}
S(x^i y^{n+2-i}, g) &= - s x^{i-1} y^{n+3-i} - x^{i+1} y^{n+1-i} + x^{i-1} y^{n+3-i} \\
&= (-s+1) x^{i-1} y^{n+3-i} + (-1)x^{i+1}y^{n+1-i}.
\end{align*}
\item $S(x^{n-1} y^3, g) = -s x^{n-2}y^4 - x^n y^2 + x^{n-2} y^4 = (-s+1) x^{n-2} y^4 - (y^2) x^{n}$.
\item $S(x^{n}, g) = -s x^{n-1} y^3 - x^{n+1}y + x^{n-1} y^3 = (-s+1) x^{n-1} y^3 + (xy) x^{n}$.
\end{itemize}
Thus, $G$ is a Gr\"obner basis of $\ie$.  The leading term $x^{n-1} y^2$ of $f$ is manifestly not divisible by any of the leading terms of $G$, which means that the output of the division algorithm of $f$ by $G$ is $f$.  Thus, $f \notin \ie$, as required.

To demonstrate (5), we will use Theorem~\ref{thm:elim}.  Accordingly, let $B := k[r,s,x,y]$, ordered lexicographically with $r>s>x>y$, and consider the ideal $\ia := r\ih B+ (1-r) sB$ of $B$.  We claim that the entries of the following vector comprise a Gr\"obner basis of $\ia$.  Note that it ends with all the elements of $s G \cup \{sf\}$ except for $s x^{n-1}y^3$.
\[
\left[ \begin{matrix}
rs-s\\
r x^{n}\\
c := r x^3y - r x y^3 - s x^2 y^2 + s x y^3 \\
d := m r x^2 y^{n-1} + \sum_{j=1}^{n-3} (-1)^{j-1} j s x^{n-1-j} y^{j+2} \\
r y^{n} \\
-sg = s^2 x^2 y^2 - s^2 xy^3 - s x^3y + s x y^3\\
s x^{n} \\
s f = \sum_{j=2}^{n-1} (-1)^j s x^{n+1-j}y^j \\
s x^{n-2} y^4 \\
s x^{n-3} y^5 \\
\vdots \\
s x^3 y^{n-1} \\
s y^{n}
\end{matrix} \right]
\]
(This is a vector of length $n+5$, and we label the elements $F_0$ through $F_{n+4}$.) First we have to show that the ideal generated by the entries of $F$ is exactly $\ia$.  To see that $\ia \subseteq (F)$, \begin{itemize}
\item $rg = (-x^2 y^2 + x y^3) (rs-s) + 1\cdot c$, and
\item $rf = (x-y) \left(\displaystyle \sum_{j=1}^{m-1} j x^{n-3-2j} y^{2j-1}\right)c + d + mx (-r y^{n} + s y^{n})$.
\end{itemize}
To see that $F \subseteq \ia$, \begin{itemize}
\item $c = 1\cdot (rg) + (-x^2 y^2 + x y^3) (-rs+s)$,
\item $d =(m-1)(-sx+x)(ry^{n}) + (-x+y)\left(\displaystyle \sum_{j=1}^{m-1} j x^{n-3-2j} y^{2j-1}\right)(rg) + 1\cdot (rf)+ \left(\displaystyle \sum_{j=1}^{n-3} (-1)^{j-1} jx^{n-j-1}y^{j+2}\right)(-rs+s)$,
\end{itemize}
and for each element $u \in G \cup \{f\}$, we have $ru \in \ia$, so that \begin{itemize}
\item $su = s\cdot (ru) + u \cdot(-rs+s)$.
\end{itemize}
Thus, $\ia = (F)$.

Taking all the S-polynomials $S_{jk} = S(F_j, F_k)$ such that $j<k$ and $F_j$, $F_k$ are not both monomials (and note that the only non-monomials are $F_i$ for $i=0,2,3,5,7$), we may obtain the following list.  For these choices of $a_{ijk}$, the diligent reader may easily verify the conditions of Theorem~\ref{thm:Bucrit}:
\begin{itemize}
\item $S_{01} = -F_6$
\item $S_{02} = x y^3 F_0 + F_5$
\item $S_{03} = ((-x+y)\sum_{j=1}^{m-1} j x^{n-3-2j} y^{2j-1})F_5  - F_7 + (m-1)(sx -x) F_{n+4}$
\item $S_{04} = -F_{n+4}$
\item $S_{05} = (s x y^3 + x^3y-xy^3)F_0 - F_5$
\item $S_{06} = -F_6$
\item $S_{07} = (\sum_{j=3}^{n-1} (-1)^{j-1} x^{n-j+1}y^j) F_0 - F_7$
\item $S_{0i} = -F_i$, for $8\leq i \leq n+4$
\item $S_{12} = (\sum_{j=1}^{m-1} x^{n-3-2j} y^{2j}) F_2 + xF_4 + F_7 - x F_{n+4}$
\item $S_{13} = (\sum_{j=1}^{n-3} (-1)^j j x^{n-3-j} y^{j+2})F_6$
\item $S_{15} =  (\sum_{j=1}^{n-2} r x^{n-2-j} y^j)F_5  + (rxy+ry^2)F_6 + (rsxy-rx^2 - rxy)F_{n+4}$
\item $S_{17} = - sx^2 F_4+ryF_7$
\item $S_{23} = -mxyF_4 - y F_7 + (\sum_{j=1}^{n-4} (-1)^{j-1} jF_{j+7}) + ((1-m)x^2 + mxy)F_{n+4}$
\item $S_{24} = -x y^2 F_4 + (-x^2 y + x y^2) F_{n+4}$
\item $S_{25} = (s x^2 y^3 - s x y^4 + x^4 y - x^2 y^3) F_0 + (-sy-x)F_5$
\item $S_{26} = -s(\sum_{j=1}^{m-1} x^{n-3-2j}y^{2j})F_2 - sxF_4 - sF_7 + sxF_{n+4}$
\item $S_{27} = (s x^{n-5} y^2 + (-x+y)\sum_{j=2}^{m-1} jsx^{n-3-2j}y^{2j-1})F_2 - s F_3 + (m-1)(rx-sx)F_{n+4}$
\item $S_{2i} =  -sF_{i+1} + (-r+s) F_{i+2}$, for $8\leq i \leq n+1$
\item $S_{2,n+2} = -sF_{n+3} + (-r+s)x^2 F_{n+4}$
\item $S_{2,n+3} = (-rxy-sx^2+sxy)F_{n+4}$
\item $S_{2,n+4} = (-rxy^2 - sx^2y + s xy^2)F_{n+4}$
\item $S_{34} = (\sum_{j=1}^{n-4}(-1)^{j-1} j F_{j+7}) - (n-3) x^2 F_{n+4}$
\item $S_{35} = m(sxy^n + x^3 y^{n-2} - xy^n)F_0 + [s(x-y)(\sum_{j=1}^{m-1} j x^{n-3-2j} y^{2j-1}) - m y^8] F_5 +sF_7 + (m-1) (-s^2 x + s x) F_{n+4}$
\item $S_{36} = (\sum_{j=1}^{n-3} (-1)^{j-1} j x^{n-j-3}y^{j+2}) F_6$
\item $S_{37} = ms(\sum_{j=0}^{n-4} (-1)^j x^{n-2-j}y^j) F_4 + (\sum_{j=1}^{n-4} (-1)^{j-1} jsx^{n-4-j}y^{j+2})F_6 - (n-3)sxy^{n-5} F_8$
\item $S_{3i} = x^{n+3-i} (sy F_7 + s(\sum_{j=1}^{n-4} (-1)^j jF_{j+7}) - sx^2 F_{n+4})$ for $8\leq i\leq n+3$
\item $S_{3,n+4} = (\sum_{j=1}^{n-4} (-1)^{j-1}js F_{j+7}) - (n-3) sx^2 F_{n+4}$
\item $S_{45} = (s^2 xy-sxy) F_4 + r F_{n+3}$
\item $S_{47} = (\sum_{j=1}^{n-3} (-1)^{j-1} s x^{n-1-j} y^j) F_4$
\item $S_{56} = -(\sum_{j=1}^{n-2} x^{n-2-j} y^j)F_5 -(xy+y^2)F_6 + (-sxy+x^2+xy)F_{n+4}$
\item $S_{57} = -(\sum_{j=1}^{m-1} x^{n-3-2j} y^{2j}) F_5 -yF_6+(-sx+x)F_{n+4}$
\item $S_{58} = -yF_7 - F_8 - (s+2)F_9 + (\sum_{j=10}^{n+3} (-1)^{j-1}F_j)+ x^2 F_{n+4}$
\item $S_{5i} = - F_{i-1} + (-s+1)F_{i+1}$ for $9\leq i\leq n+2$
\item $S_{5,n+3} =  -F_{n+2}+ (-s+1)x^2 F_{n+4}$
\item $S_{5,n+4} = - F_{n+3} + (-s+1)xyF_{n+4}$
\item $S_{67} = y F_7 - x^2 F_{n+4}$
\item $S_{7i} = xy^{i-8} ((\sum_{j=9}^{n+3} (-1)^j F_j) + (-x^2+xy)F_{n+4})$ for $8\leq i\leq n+3$
\item $S_{7,n+4} = (\sum_{j=1}^{n-3} (-1)^j x^{n-1-j} y^j)F_{n+4}$
\end{itemize}
Hence, the entries of $F$ give a Gr\"obner basis of $\ia$. By Theorem~\ref{thm:elim}, it follows that the elements of $F$ whose leading term does not involve $r$ forms a generating set for the ideal $\ih \cap (s)$ of $A$.  That is, $\ih \cap (s) = (s y^{n}, s x^3 y^4 (x,y)^{n-5}, s f, s x^{n}, sg)$.  Dividing by $s$, we get $(\ih : s) = (y^{n}, x^3 y^4(x,y)^{n-5}, f, x^{n}, g) = \ih$ (since $x^3 y^4 (x,y)^{n-5} \subseteq (x,y)^{n+2} = \ib \subseteq \ih$), as required. 

To see (6), first note that $\ie : \m^\infty = \ie : s^\infty$, since $\ie$ contains powers of both $x$ and $y$.  But $\ie \subseteq \ih$, so  from (2) and (5), we have $\ih \subseteq (\ie:s) \subseteq (\ih:s) = \ih$, whence all are equalities.  Thus, $(\ie : s^\infty) = (\ih : s^\infty) = \ih$, as required.

Finally, to see (7), it follows from (6) and (4) that \[
H^0_\m(A/\ie) = \frac{\ie : \m^\infty}{\ie} = \frac{\ih}{\ie} = \frac{\ie + (f)}{\ie} \cong \frac{A}{(\ie :f)} = A/\m.
\]
\end{construction}

\begin{example}
Let $p$ be an odd prime number.  Let $k$ be a field of characteristic $p$, and $R := k[s,x,y] / (xy (x-y)(x+y-sy))$.  This is the ring used by Katzman in \cite{KaUAss}, with variable change given by $s=t+1$.

Consider the ideals $J := (x^p, y^p)$ and $I=(x,y)^p$ of $R$.  As shown in Katzman's paper, $J^* = I$.  Now fix a power $q=p^e$ of $p$, $e\geq 1$, and let $n=pq$ in Construction~\ref{cons}.  Let $\ib$, $\ie$, $A$, $\m$, $g$, $\ih$, and $f$ be as in that construction.  The conditions of the construction are satisfied, since $p\geq 3$, whence $n= pq \geq 9$, and $p$ can never divide $(pq-1)/2$.  Then $R=A/(g)$ and $J^{[q]} = \ie/(g) \subseteq R$.  In particular, letting $z$ be the image of $f$ in $R$, we have $z\notin J^{[q]}$.  

However, we claim that $z\in I^{[q]}$.  To see this, it is enough to show (in the ring $R$ -- that is, modulo $g$) that for all $j=2, 3, \dotsc, pq-1$, we have $x^j y^{pq+1-j} \in (x^q, y^q)^p$.   For $j=q, q+1$ this is clear, and for $j\geq q+2$, the assertion follows from the equation $x^j y^{pq+1-j} = (t+1) x^{j-1} y^{pq-j+2} - t x^{j-2} y^{pq-j+3}$, along with induction, showing that in these cases, $x^j y^{pq+1-j} \in (x^q y^{pq-q})$. For $2\leq j \leq q-1$, we have \begin{align*}
x^j y^{pq+1-j} 
&= x^{q+1} y^{pq-q} - \sum_{i=1}^{j-2} x^{q+i} y^{pq-q-i}(y-x) - x^j y^{pq-q+2-j}(x^{q-1}-y^{q-1}) \\
&= x^{q+1} y^{pq-q} - \sum_{i=1}^{j-2} t^i x^{q} y^{pq-q}(y-x) - t^{j-1}x y^{pq-q+1}(x^{q-1}-y^{q-1}) \\
&= x^{q+1} y^{pq-q} - \sum_{i=1}^{j-2} t^i x^{q} y^{pq-q}(y-x) - t^{j-1}x^q y^{pq-q+1}
   + t^{j-1}x y^{pq} \\
& \in (x^q, y^q)^p = I^{[q]}.
\end{align*}

Let $\p := (x,y)$.  We claim that  $\p \in \Ass_R(I^{[q]}/ J^{[q]})$.  Since $\p$ is minimal over $J^{[q]}$, it suffices to show that $I^{[q]}_\p/ J^{[q]}_\p = (I^{[q]} / J^{[q]})_\p \neq 0$.  To do this, it suffices to show that $((x^{pq},y^{pq}) +(g))A_P$ is properly contained in $((x^q, y^q)^p + (g))A_P$, where $P := (x,y) \subseteq A$.  But in the ring $C := L[x,y]$ (where $L:=k(s)$, the fraction field of $k[s]$), the ideal $(x^{pq},y^{pq}, g)C$ is primary to $P' = (x,y)C$, which is a maximal ideal of $C$.  So to show that $((x^q, y^q)^p + (g))C_{P'} / ((x^{pq},y^{pq}) +(g))C_{P'} = ((x^q, y^q)^p + (g) / (x^{pq}, y^{pq},g))_{P'}$ is nonzero over $A_P = C_{P'}$, it suffices to show that the $C$-module $((x^q, y^q)^p + (g))C / (x^{pq},y^{pq}, g)C \neq 0$.  For this, it is enough to show that $x^q y^{(p-1)q} \notin \ic := (x^{pq},y^{pq}, xy(x-y))C$, since $xy(x-y)$ is a factor of $g$.  Suppose that $x^q y^{(p-1)q} \in \ic$.  Then there exist polynomials $a,b,c \in C$ such that \[
x^q y^{(p-1)q} = a x^{pq} + b y^{pq} + cxy(x-y).
\]
From degree considerations (taking the homogeneous degree $pq$-part of the above equation), we may assume that $a,b \in L$.  Then making the substitution ($x=1$, $y=0$) in the displayed equation yields $a=0$, whereas the substitution ($x=0$, $y=1$) yields $b=0$.   So $x^q y^{(p-1)q} = cxy(x-y)$.  But then the substitution $x=y=1$ leads to the conclusion that $1=0$, a manifest contradiction.  Hence $x^q y^{(p-1)q} \notin \ic$, so that  $\p \in \Ass_R(I^{[q]}/ J^{[q]})$, as required.

We also know from Construction~\ref{cons} that $\m = (J^{[q]} : z)$, so that since $z\in I^{[q]}$, we have $\m \in \Ass_R(I^{[q]}/J^{[q]})$ as well.  So we are in the ``critical situation'' described at the beginning of this Section.

Moreover, \[
\len_R(H^0_\m(I^{[q]}/J^{[q]})) \leq \len_A(H^0_\m(A/\ie)) = \len_A(A/\m) = 1,
\]
a constant, which shows that $\rjj{R_\m}{J_\m}{I_\m} = 0$, since $\dim R/J = 1>0$.  Hence, the expected converse to Theorem~\ref{thm:zerotc} holds for this specific example.
\end{example}

\section{Variants of relative multiplicity}\label{sec:var}

The proof of Theorem~\ref{thm:tczero-fg} suggests a slight variant on Definition~\ref{def:rjj}.  Namely:

\begin{defn}\label{def:sjj}
For modules $L \subseteq M \subseteq N$ over a local ring $(R,\m)$ of dimension $d$, we set \[
\sjjp NLM := \limsup_{q \ra \infty} \frac{\len([\Gamma_\m(M/L)]^{[q]}_{N/L})}{q^d}
\] and \[
\sjjm NLM := \liminf_{q \ra \infty} \frac{\len([\Gamma_\m(M/L)]^{[q]}_{N/L})}{q^d}.
\]
If these are equal (\emph{i.e.}, the limit is well-defined), then the common number is written $\sjj NLM$.
\end{defn}

We note that $\rjjm NLM$ (resp. $\rjjp NLM$) is an upper bound for $\sjjm NLM$ (resp. $\sjjp NLM$).  Also, this new concept has the advantage of being \emph{bounded above}:

\begin{lemma}
If $M/L$ is finitely generated, then $\sjjp NLM < \infty$.
\end{lemma}

\begin{proof}
Let $H$ be the submodule of $M$ such that $H/L = \Gamma_\m(M/L)$.  Then since $H/L$ has finite length and $H^{[q]}_N / L^{[q]}_N$ is the image of the map $F^e(H/L) \ra F^e(N/L)$, we have \[
\limsup_{q\ra \infty} \frac{\len(\Gamma_\m(M/L)^{[q]}_{N/L})}{q^d} \leq \lim_{q \ra \infty} \frac{\len(F^e(H/L))}{q^d}
\]
But since $H/L$ has finite length, this last limit is well-defined and finite by Seibert~\cite[Proposition 2]{Sei-cx}.
\end{proof}

In preparation for the next theorem, we define two more slight variants of relative multiplicity:
\begin{defn}
For modules $L \subseteq M \subseteq N$ over a local ring $(R,\m)$ of dimension $d$, we set \[
\tjjp NLM := \sup \{\rjjp N T M \mid L \subseteq T \subseteq M \text{ and } \len(M/T)<\infty\}
\] and \[
\tjjm NLM := \sup \{\rjjm N T M \mid L \subseteq T \subseteq M \text{ and } \len(M/T)<\infty\}.
\] Finally, we set \[
\ujjm NLM := \sjjm N {L^*_N \cap M} M.
\]
\end{defn}

\begin{thm}\label{thm:var}
Let $R$ be a Noetherian ring of prime characteristic $p>0$, and let $L \subseteq M \subseteq N$ be finitely generated $R$-modules.  Consider the following conditions: \begin{enumerate}[label=(\alph*)]
\item $M \subseteq L^*_N$.
\item $\tjjp {N_\p} {L_\p} {M_\p} = 0$ for all $\p \in \Spec R$.
\item $\tjjm {N_\p}{L_\p}{M_\p} = 0$ for all $\p \in \Spec R$.
\item $\ujjm {N_\p}{L_\p}{M_\p} = 0$ for all $\p$ such that $\len(M_\p / ((L_\p)^*_{N_\p} \cap M_\p)) < \infty$.
\end{enumerate}
Then (a) $\implies$ (b) $\implies$ (c) $\implies$ (d).  If, moreover, $R$ contains a completely stable weak test element, $\widehat{R_\p}$ is reduced and equidimensional for all $\p \in \Spec R$, and tight closure commutes with localization for the pair $L \subseteq N$, then (d) $\implies$ (a).
\end{thm}

\begin{proof}
To see that (a) $\implies$ (b), note that for any module $T$ between $L$ and $M$, we have $M_\p \subseteq (T^*_N)_\p \subseteq (T_\p)^*_{N_\p}$, so that when $\len(M_\p / T_\p) < \infty$ over $R_\p$, Theorem~\ref{thm:HHlen} shows that $\rjjp {N_\p}{T_\p}{M_\p} = 0$.  Since this holds for all such $T$, we have $\tjjp {N_\p}{L_\p}{M_\p} = 0$.  We have (b) $\implies$ (c) $\implies$ (d) by the definitions.

We need only show that the additional conditions require that (d) $\implies$ (a).  We will prove the contrapositive.  That is, suppose that $M \nsubseteq L^*_N$.  Then let $T := L^*_N \cap M$.  Since $T \subsetneq M$, there is some minimal prime $\p$ of $M/T$, which means that $0<\len(M_\p / T_\p)< \infty$ as $R_\p$-modules.  Moreover, $T_\p = (L^*_N)_\p \cap M_\p = (L_\p)^*_{N_\p} \cap M_\p$.  Since $M_\p \nsubseteq (L^*_N)_\p = (L_\p)^*_{N_\p} \supseteq (T_\p)^*_{N_\p}$, it follows that $M_\p \nsubseteq (T_\p)^*_{N_\p}$.  Then \[
\ujjm {N_\p}{L_\p}{M_\p} = \rjjm {N_\p}{T_\p}{M_\p} = \liminf_{q\ra \infty} \frac{\len((M_\p)^{[q]}_{N_\p} / (T_\p)^{[q]}_{N_\p})}{q^{\height \p}} >0,
\]
where the last inequality is by Theorem~\ref{thm:HHlen}.
\end{proof}

We think of this theorem as an avatar of the fact \cite[Proposition 6.1]{HHmain} that every tightly closed ideal is an intersection of finite colength tightly closed ideals.

It is natural to ask whether Definition~\ref{def:rjj} and Definition~\ref{def:sjj} are equivalent.  In fact they are not:

\begin{prop}\label{prop:reg}
Let $R$ be a Noetherian ring of prime characteristic $p>0$.  Consider the following conditions: \begin{enumerate}[label=(\alph*)]
\item $R$ is regular.
\item $\rjj {N_\p}{L_\p}{M_\p} = \sjj {N_\p}{L_\p}{M_\p}$ for all submodule inclusions $L \subseteq M \subseteq N$ and all $\p \in \Spec R$.
\item $\rjj {N_\p}{0}{N_\p} = \sjj{N_\p} 0 {N_\p}$ for all finite $R$-modules $N$ and all $\p \in \Spec R$.
\end{enumerate}
Then (a) $\implies$ (b) $\implies$ (c), and if $R$ is reduced and has finite F-representation type then (c) $\implies$ (a).
\end{prop}

\begin{proof}
(a) $\implies$ (b) because when $R$ is regular, the functors $F^e$ and $H^0_\m$ commute with each other.  Obviously (b) $\implies$ (c).

So suppose (c) holds, and assume $R$ is reduced and has FFRT.  We want to show $R_\m$ is regular for all maximal ideals $\m$, so we may replace $R$ by $R_\m$ for a maximal ideal $\m$, and let $d = \dim R = \height \m$.  We adopt the notation and terminology from Theorem~\ref{thm:Yao}.  For a fixed finite $R$-module $N$ and prime ideal $\p$, we have \[
\rjj {N_\p}0{N_\p} = \sum_{i=1}^r \len_{R_\p} (H^0_{\p R_\p} (N_\p \otimes (M_i)_\p)) \cdot \ell_i,
\] whereas \[
\sjj {N_\p}0{N_\p} = \sum_{i=1}^r \len_{R_\p} (\im ( H^0_{\p R_\p}(N_\p) \otimes (M_i)_\p \longrightarrow N_\p \otimes (M_i)_\p )) \cdot \ell_i.
\]
The fact that these are equal amounts to saying that $\len_{R_\p} (H^0_{\p R_\p} (N_\p \otimes (M_i)_\p)) = \len_{R_\p} (\im ( H^0_{\p R_\p}(N_\p) \otimes (M_i)_\p \ra N_\p \otimes (M_i)_\p ))$ for each $1 \leq i \leq r$. In particular, if $\p \in \Ass(N \otimes M_i)$, then $\len_{R_\p} (H^0_{\p R_\p} (N_\p \otimes (M_i)_\p))  \neq 0$, whence $\im ( H^0_{\p R_\p}(N_\p) \otimes (M_i)_\p \ra N_\p \otimes (M_i)_\p ) \neq 0$, which implies that $H^0_{\p R_\p}(N_\p) \neq 0$, so that $\p \in \Ass N$.

That is, for all finitely generated $R$-modules $N$ and all F-contributors $M_i$, we have $\Ass (N \otimes M_i) \subseteq \Ass N$.  But the authors have shown in \cite{nmeYao-flat} that any $R$-module $V$ such that $\Ass (N \otimes V) \subseteq \Ass N$ for all finite $N$ must be flat, provided that $R$ is reduced.  Thus, each $M_i$ is flat, hence (since they are finitely generated) free.  Thus, we may arrange it (by re-grouping the summands) so that \emph{$R$ is the only F-contributor!}  That is, $r=1$ and $M_1 = R$.

Hence, \begin{align*}
1 &\leq \eHK(\m) = \lim_{q \ra \infty} \frac{\len(R/\m^{[q]})}{q^d} = \lim_{e \ra \infty} \frac{\len((R/\m) \otimes_R {}^eR)}{(a p^d)^e}\\
&= \sum_{i=1}^r \left(\lim_{e \ra \infty} \frac{c_{i,e}}{(a p^d)^e}\right) \cdot \len(R/\m \otimes_R M_i) = \ell_1 \cdot \len(R/\m) = \ell_1.
\end{align*}
However, note that $\ell_1$ is the \emph{F-signature} of the ring $R$ (cf. Huneke and Leuschke \cite{HunLeu-2t}), so that by \cite[Theorem 11 and Proposition 14]{HunLeu-2t}, $\ell_1 \leq 1$.  Thus, $\ell_1 = 1$, and then \cite[Corollary 16]{HunLeu-2t} shows that $R$ is regular.
\end{proof}

\begin{rmk*}
In~\cite{nmeYao-flat}, the authors have in fact shown that if $R$ is reduced and $V$ is an $R$-module, then $V$ is flat if and only if $\Ass(Q \otimes V) \subseteq \Ass Q$ for all \emph{prime ideals} $Q$.  Given this, the proof of Proposition~\ref{prop:reg} yields a stronger result.  Namely, if $R$ is a reduced Noetherian ring of prime characteristic and finite F-representation type, then $R$ is regular if and only if $\rjj {Q_\p}{0}{Q_\p} = \sjj{Q_\p} 0 {Q_\p}$ for all $\p, Q \in \Spec R$.
\end{rmk*}

\section{A numerical criterion based on a Nakayama-type lemma}\label{sec:Nak*}
If all we wanted to do was to get a numerical criterion determining exactly when two modules have the same tight closure, it already exists, in view of the following result of the first named author, for which we provide a new, slightly simplified proof here in order to make the paper more self-contained:

\begin{prop}[Nakayama lemma for tight closure]\cite[Corollary 3.2]{nmepdep}\label{pr:Nak}
Let $(R,\m)$ be a Noetherian local ring of prime characteristic $p>0$ which possesses a weak test element (\emph{e.g.} this holds whenever $R$ is excellent \cite[Theorem 6.1a]{HHbase}).  Let $L \subseteq M \subseteq N$ be finitely generated $R$-modules such that $L \subseteq M \subseteq (L + \m M)^*_N$.  Then $M \subseteq L^*_N$.
\end{prop}

\begin{proof}
First, let $q_0$ be some power of $p$ such that there exists a $q_0$-weak test element $c$.  We show by induction on $r$ that $M \subseteq (L + m^r M)^*_N$ for all $r\geq 1$.

The case $r=1$ holds by assumption.  So assume inductively that $r>1$ and $M \subseteq (L + \m^{r-1} M)^*_N$.  From now on all bracket powers (except on $\m$) are taken as submodules of $N$.  Since $M$ is finitely generated, it follows that for all $q\geq q_0$, we have $c M^{[q]} \subseteq (L + \m^{r-1} M)^{[q]}$, so that $c^2 M^{[q]} \subseteq L^{[q]} + (\m^{[q]})^{r-1} c M^{[q]} \subseteq L^{[q]} + (\m^{[q]})^{r-1} (L +\m M)^{[q]} \subseteq L^{[q]} + (\m^{[q]})^r M^{[q]} = (L + \m^r M)^{[q]}$.  Since this holds for all $q\geq q_0$, it follows that $M \subseteq (L + \m^r M)^*_N$, completing the induction.

Now fix any power $q\geq q_0$ of $p$.  Since $M \subseteq (L + \m^r M)^*_N$ for all $r\geq 1$, we have that $c M^{[q]} \subseteq L^{[q]} + (\m^{[q]})^r M^{[q]}$ for all $r$ (for this particular $q$), so that 
$c M^{[q]}  \subseteq \bigcap_{r\geq 1} L^{[q]} + (\m^{[q]})^r M^{[q]}$.  Now going mod $L^{[q]}$ and taking bracket powers in $N/L$, we have that $c (M/L)^{[q]} \subseteq \bigcap_{r\geq 1} (\m^{[q]})^r (M/L)^{[q]} = 0$, by the Krull intersection theorem.  Now `unfix' $q$, so that we have $c M^{[q]}_N \subseteq L^{[q]}_N$ for all $q\geq q_0$, whence $M \subseteq L^*_N$.
\end{proof}

Now define \[
\vjjm NLM := \liminf_{q \ra \infty} \frac{\len(M^{[q]}_N / (L + \m M)^{[q]}_N)}{q^d}.
\]
and let $\vjjp NLM$ be the corresponding $\limsup$.

\begin{prop}
Let $(R,\m)$ be a Noetherian local ring of prime characteristic $p>0$ with a weak test element.  Let $L \subseteq M \subseteq N$ be finitely generated $R$-modules.  The following are equivalent: \begin{enumerate}
\item $M \subseteq L^*_N$.
\item $\vjjp NLM =0$.
\item $\vjjm NLM = 0$.
\end{enumerate}
\end{prop}

\begin{proof}
By Proposition~\ref{pr:Nak}, $M \subseteq L^*_N \iff M \subseteq (L + \m M)^*_N$.  By Theorem~\ref{thm:HHlen} and since $\len(M / (L + \m M)) \leq \len(M/\m M) = \mu(M)<\infty$, $M \subseteq (L + \m M)^*_N \iff \vjjp NLM =0 \iff \vjjm NLM = 0$.
\end{proof}

Here is a global version:

\begin{prop}
Let $R$ be a Noetherian ring of prime characteristic $p>0$ which is reduced, locally equidimensional, and essentially of finite type over an excellent local ring.  Let $L \subseteq M \subseteq N$ be finitely generated $R$-modules.  The following are equivalent:
\begin{enumerate}[label=(\alph*)]
\item $M \subseteq L^*_N$.
\item $M \subseteq (L + \m M)^*_N$ for all maximal ideals $\m$.
\item $M_\m \subseteq (L_\m + \m M_\m)^*_{N_\m}$ for all maximal ideals $\m$.
\item $\vjjp {N_\m} {L_\m}{M_\m} = 0$ for all maximal ideals $\m$.
\item $\vjjm {N_\m} {L_\m}{M_\m} = 0$ for all maximal ideals $\m$.
\item $M_\m \subseteq (L_\m)^*_{N_\m}$ for all maximal ideals $\m$.
\end{enumerate}
\end{prop}

\begin{proof}
(a) $\implies$ (b) $\implies$ (c): Clear.

(c) $\iff$ (d) $\iff$ (e): Since $R$ is excellent, each $\widehat{R_\m}$ is still equidimensional and reduced. Then the equivalence follows from Theorem~\ref{thm:HHlen} applied to each $R_\m$.

(c) $\implies$ (f), by Proposition~\ref{pr:Nak}, since $R$ has a locally stable weak test element by \cite[Theorem 6.1]{HHbase}.

(f) $\implies$ (a): Let $c\in R^\circ$ be a locally stable $q_0$-weak test element.  Fix $q \geq q_0$. Then for all maximal ideals $\m$, we have \[
(c M^{[q]}_N)_\m = \frac{c}{1} \cdot (M_\m)^{[q]}_{N_\m} \subseteq (L_\m)^{[q]}_{N_\m}
=(L^{[q]}_N)_\m.
\]
Since containment is a local property, it follows that $c M^{[q]}_N \subseteq L^{[q]}_N$.  Since this holds for all $q\geq q_0$, it follows that $M \subseteq L^*_N$.
\end{proof}

\section{\texorpdfstring{Another numerical characterization of tight closure when $\len(I/J)<\infty$}
{Another numerical characterization of tight closure}}\label{sec:jHK}

Inspired by j-multiplicity, we make the following definitions:

\begin{defn}
For an ideal $K$ of $R$ and an integer $e \geq -1$, (using the convention $K^{[p^{-1}]} := R$) we set \[
l_e(K) := \len(\Gamma_\m(K^{[p^e]} / K^{[p^{e+1}]}))
\]
and \[
f_e(K) := \sum_{n=-1}^{e-1} l_n(K).
\]
For a pair of ideals $J \subseteq I$, set \[
\qjjm(J, I) := \liminf_{q \ra \infty} \frac{f_e(J) - f_e(I)}{q^d},
\]
and \[
\qjjp(J, I) := \limsup_{q \ra \infty} \frac{f_e(J) - f_e(I)}{q^d}.
\]
If these two quantities are equal, we denote the common number by $\qjj(J, I)$.
\end{defn}

We have the following:

\begin{thm}\label{thm:jHKcrit}
Let $(R,\m)$ be a Noetherian local ring of dimension $d$ and prime characteristic $p>0$, and let $J$, $I$ be ideals such that $J \subseteq I$ and $\len(I/J)<\infty$.  Consider the following conditions: \begin{enumerate}[label=(\alph*)]
\item $I^*=J^*$.
\item $\displaystyle \qjj(J, I) = 0.$
\item $\displaystyle \qjjm(J, I) \leq 0.$
\end{enumerate}
Then (a) $\implies$ (b) $\implies$ (c).  If moreover $R$ has a completely stable weak test element and $\hat{R}$ is equidimensional and reduced, or if $\dim R=0$, then (c) $\implies$ (a) as well.
\end{thm}

\begin{proof}
We dispense first with the case where $\dim R=0$.  In this case, for any proper ideal $K$ we have $K^{[q]} = 0$ for $q\gg 0$, and hence $f_n(K) = \len(R)>0$ for $n \gg 0$, whereas $f_n(R) = 0$ for all $n$.  Also, $I^*=J^*$ if and only if both ideals are proper or both are the unit ideal.  If both ideals are proper, then $f_n(I) = \len(R) = f_n(J)$ for $n\gg 0$, whence (b) holds.  If both ideals are improper, then $I=J=R$, so that (b) holds.  So we see that (a) $\implies$ (b).  Conversely, suppose that (c) holds.  If $I=R$ then $f_n(I)=0$, whence $f_n(J) =0$ for infinitely many values of $n$, which forces $J=R$.  On the other hand, if $I$ is proper, then $J$ is proper since $J \subseteq I$.  In either case, (a) holds, so (c) $\implies$ (a).

From now on, we assume that $d = \dim R>0$.  First note that we have the following short exact sequences for all $q=p^e$, $e\geq -1$: \begin{equation}\label{seq:1q}
0 \ra I^{[pq]}/J^{[pq]} \ra I^{[q]}/J^{[pq]} \ra I^{[q]} / I^{[pq]} \ra 0
\end{equation}
and \begin{equation}\label{seq:2q}
0 \ra J^{[q]} / J^{[pq]} \ra I^{[q]} / J^{[pq]} \ra I^{[q]} / J^{[q]} \ra 0.
\end{equation}
Applying $\Gamma_\m$ to sequence (\ref{seq:1q}), and using the fact that $I^{[pq]} / J^{[pq]}$ has finite length, we get the short exact sequence: \begin{equation}\label{seq:1q'}
0 \ra I^{[pq]}/J^{[pq]} \ra \Gamma_\m(I^{[q]}/J^{[pq]}) \ra \Gamma_\m(I^{[q]} / I^{[pq]}) \ra 0.
\end{equation}
Hence, \begin{equation}\label{eq:ij}
\len(\Gamma_\m(I^{[q]}/J^{[pq]})) = \len(I^{[pq]}/J^{[pq]}) + l_e(I).
\end{equation}
Now, applying $\Gamma_\m$ to the sequence (\ref{seq:2q}) and using the fact that $I^{[q]}/J^{[q]}$ has finite length, we get the following exact sequence: \begin{equation}\label{seq:2q'}
0 \ra \Gamma_\m(J^{[q]}/J^{[pq]}) \ra \Gamma_\m(I^{[q]}/J^{[pq]}) \ra I^{[q]} / J^{[q]},
\end{equation}
which leads to the inequalities: \begin{equation}\label{ineq:ij}
l_e(J) \leq \len(\Gamma_\m(I^{[q]}/J^{[pq]})) \leq l_e(J) + \len(I^{[q]}/J^{[q]}).
\end{equation}
Combining Equation~\ref{eq:ij} with Inequalities~\ref{ineq:ij}, we get: 
\[
l_e(J) \leq \len(I^{[pq]} / J^{[pq]}) + l_e(I) \leq l_e(J) + \len(I^{[q]}/J^{[q]}),
\]
which are equivalent to the following: \begin{equation}\label{ineq:q}
\len(I^{[pq]}/J^{[pq]}) - \len(I^{[q]} /J^{[q]}) \leq l_e(J)-l_e(I) \leq \len(I^{[pq]}/J^{[pq]}).
\end{equation}
Taking the sum of Inequalities~\ref{ineq:q} from $e=-1$ to $n-1$, we get: \begin{equation}\label{ineq:total}
\len(I^{[p^n]}/J^{[p^n]}) \leq f_n(J) - f_n(I) \leq \sum_{j=0}^n \len(I^{[p^j]}/J^{[p^j]}).
\end{equation}

Following these preliminaries, we proceed to the implications in the proof.  It is obvious that (b) implies (c).  So suppose that (a) is true.  Then by Theorem~\ref{thm:HHlen}, there is a constant $C$ such that $\len(I^{[q]}/J^{[q]}) \leq C q^{d-1}$ for all $q=p^e$, $e\geq -1$.  Thus, \[
\sum_{j=0}^n \len(I^{[p^j]}/J^{[p^j]}) \leq C \cdot \sum_{j=0}^n p^{j(d-1)} = \frac{C (p^{(n+1)(d-1)}-1)}{p^{d-1}-1} < C'p^{n(d-1)},
\]
where $C':=Cp^{d-1}/(p^{d-1}-1)$. Combining with Inequalities~\ref{ineq:total} and dividing by $p^{nd}$, we get \[
0 \leq \frac{\len(I^{[p^n]} / J^{[p^n]})}{p^{nd}} \leq \frac{f_n(J) - f_n(I)}{p^{nd}} < \frac{C'p^{nd-n}}{p^{nd}} = \frac{C'}{p^n}.
\]
Since both the leftmost and rightmost terms clearly have a limit of 0 as $n \rightarrow \infty$, statement (b) follows.

Conversely, suppose $R$ satisfies the additional specified conditions and that (c) holds. Using (c) and Inequalities~\ref{ineq:total}, we have: \[
0 \leq \liminf_{n\ra \infty} \frac{\len(I^{[p^n]}/J^{[p^n]})}{p^{nd}} \leq\liminf_{n\ra \infty} \frac{f_n(J)-f_n(I)}{p^{nd}} \leq 0
\]

Thus, $\displaystyle \liminf_{n\ra \infty} \frac{\len(I^{[p^n]}/J^{[p^n]})}{p^{nd}} =0$, so by Theorem~\ref{thm:HHlen}, $I^*=J^*$.
\end{proof}

We also get a global version:

\begin{thm}\label{thm:jHKglobal}
Let $R$ be a Noetherian ring of prime characteristic $p>0$ which is F-regular on the punctured spectrum, and let $J \subseteq I$ be ideals.  Consider the following conditions: \begin{enumerate}[label=(\alph*)]
\item $(I_\m)^* = (J_\m)^*$ for all maximal ideals $\m$.
\item $\qjj(J_\p, I_\p) = 0$ for all $\p \in \Spec R$.
\item $\qjjm(J_\p, I_\p) \leq 0$ for all $\p \in \Spec R$.
\end{enumerate}
Then (a) $\implies$ (b) $\implies$ (c).  If moreover $R$ has a completely stable weak test element and $\widehat{R_\p}$ is equidimensional and reduced for all $\p \in \Spec R$, or if $\dim R=0$, then (c) $\implies$ (a) as well.
\end{thm}

\begin{proof}
First we show that (a) $\implies$ (b): First, suppose $\p$ is non-maximal, and choose a maximal ideal $\m$ such that $\p \subseteq \m$.  Then $J_\p \subseteq I_\p = (I_\m)_\p \subseteq ((J_\m)^*)_\p \subseteq ((J_\m)_\p)^* = J_\p$.  That is, $J_\p = I_\p$ for all non-maximal ideals $\p$, so it follows that $I/J$ has finite length and that (b) holds for non-maximal $\p$.  However, for any maximal ideal $\m$, since $I_\m / J_\m$ has finite length, the implication follows for maximal ideals by Theorem~\ref{thm:jHKcrit}.

Next, we show that (c) $\implies$ (a) under the stated conditions.  We first show that $I_\p = J_\p$ for all non-maximal ideals $\p$.  By Noetherian induction, we may assume that $I_\q = J_\q$ for all prime ideals $\q \subsetneq \p$.  Thus, $\len(I_\p/J_\p) <\infty$, so that Theorem~\ref{thm:jHKcrit} applied to the ideals $J_\p \subseteq I_\p$ shows that $(J_\p)^* = (I_\p)^*$.  But $R_\p$ is F-regular, so $J_\p = I_\p$.

We have now that $\len(I/J) < \infty$, so that for any maximal ideal $\m$, Theorem~\ref{thm:jHKcrit}  shows that $(J_\m)^* = (I_\m)^*$.
\end{proof}

\section*{Acknowlegements}
The authors benefitted from conversations with many people in preparing this manuscript.  In particular, we wish to thank Ezra Miller and Kirsten Schmitz for discussions regarding Section~\ref{sec:gb}.  We used Macaulay 2 \cite{M2hyper} for some computations in that section.

\providecommand{\bysame}{\leavevmode\hbox to3em{\hrulefill}\thinspace}
\providecommand{\MR}{\relax\ifhmode\unskip\space\fi MR }
\providecommand{\MRhref}[2]{%
  \href{http://www.ams.org/mathscinet-getitem?mr=#1}{#2}
}
\providecommand{\href}[2]{#2}

\end{document}